\documentclass[12pt, a4paper]{amsart}

\usepackage{amsmath, amssymb, amsfonts, amsthm}
\usepackage{hyperref}
\usepackage{fullpage}
\usepackage{stmaryrd}
\usepackage[usenames,dvipsnames]{color}
\usepackage[all]{xy}
\usepackage{comment}
\usepackage{aliascnt}

\usepackage{pictexwd,dcpic}
\usepackage{graphicx}

\theoremstyle{plain}
\newtheorem{theorem}{Theorem}[section]

\newaliascnt{proposition}{theorem}
\newaliascnt{lemma}{theorem}
\newaliascnt{corollary}{theorem}
\newaliascnt{fact}{theorem}
\newaliascnt{observation}{theorem}
\newaliascnt{definition}{theorem}
\newaliascnt{example}{theorem}
\newaliascnt{question}{theorem}
\newaliascnt{remark}{theorem}
\newaliascnt{property}{theorem}
\newaliascnt{construction}{theorem}
\newaliascnt{setting}{theorem}

\theoremstyle{plain}
\newtheorem{proposition}[proposition]{Proposition}
\newtheorem{lemma}[lemma]{Lemma}
\newtheorem{corollary}[corollary]{Corollary}
\newtheorem{claim}{Claim}[theorem]
\newtheorem{fact}[fact]{Fact}

\theoremstyle{definition}
\newtheorem{definition}[definition]{Definition}
\newtheorem{example}[example]{Example}
\newtheorem{question}[question]{Question}
\theoremstyle{remark}

\aliascntresetthe{proposition}
\aliascntresetthe{lemma}
\aliascntresetthe{corollary}
\aliascntresetthe{fact}
\aliascntresetthe{observation}
\aliascntresetthe{definition}
\aliascntresetthe{example}
\aliascntresetthe{question}
\aliascntresetthe{remark}
\aliascntresetthe{property}
\aliascntresetthe{construction}
\aliascntresetthe{setting}

\usepackage{centernot}
\usepackage{float}
\restylefloat{figure}
\usepackage{mdframed}
\newmdenv[]{siderules}

\title{Henselianity in the language of rings}
\author{Sylvy Anscombe and Franziska Jahnke}
\thanks{\today
\\This work was begun while the first author was funded by EPSRC grant EP/K020692/1.}
\address{Jeremiah Horrocks Institute, Leighton Building Le7, University of Central Lancashire, Preston PR1 2HE, United Kingdom}
\email{sanscombe@uclan.ac.uk}
\address{Institut f\"{u}r Mathematische Logik und Grundlagenforschung,
University of M\"{u}nster,
Einsteinstr. 62,
48149 M\"{u}nster,
Germany}
\email{franziska.jahnke@wwu.de}

\begin{document}
\begin{abstract}
We consider four properties 
of a field $K$ related to the existence of (definable) henselian 
valuations on $K$ and on elementarily equivalent fields and study the 
implications between them. Surprisingly, the full pictures look very 
different in equicharacteristic and mixed characteristic.
\end{abstract}
\maketitle

\section{Introduction}

The study of henselian fields in the language of rings started with a work by Prestel and Ziegler (\cite{Prestel-Ziegler78})
where they introduced and discussed $t$-henselian fields. A  
field is \emph{$t$-henselian} if it is $\mathcal{L}_\mathrm{ring}$-elementarily
equivalent to some \emph{henselian} field, i.e., a field admitting a nontrivial henselian valuation. 
In particular, they showed that non-henselian $t$-henselian fields exist. These results are
strongly linked to the question of which fields interpret nontrivial henselian valuations in the language of rings, or
equivalently, which fields admit a nontrivial 
definable henselian valuation. Here, we say that a valuation $v$ is \emph{definable} on a field $K$
if its valuation ring $\mathcal{O}_v$ is an $\mathcal{L}_\mathrm{ring}$-definable subset of $K$ (possibly with
parameters from $K$)
and that $v$ is $\emptyset$-definable if it is definable and no parameters were needed in the defining formula. 
Henselianity is an elementary property of valued fields, in particular, it is preserved under elementary equivalence
in the language $\mathcal{L}_\mathrm{val}=\mathcal{L}_\mathrm{ring}\cup \{\mathcal{O}\}$ where the unary relation symbol
$\mathcal{O}$ is interpreted as the valuation ring. Thus, 
if some nontrivial henselian valuation ring is a $\emptyset$-definable subring of $K$,
then any $L$ which is $\mathcal{L}_\mathrm{ring}$-elementarily equivalent to $K$ also admits a nontrivial henselian
valuation. In particular, if $K$ is henselian and some $\mathcal{L}_\mathrm{ring}$-elementarily
equivalent $L$ is non-henselian, then $K$ cannot admit a $\emptyset$-definable nontrivial henselian valuation. 
Under which conditions fields admit definable nontrivial henselian valuations (with or without parameters) 
has been investigated in a number of (mostly) recent papers (\cite{Hon14}, 
\cite{Jahnke-Koenigsmann15}, \cite{Jahnke-Koenigsmann14}, \cite{Koe94}, \cite{Pr14}) and 
some of these results have been applied in connection
with the Shelah-Hasson conjecture on NIP fields 
(see \cite{JSW}, \cite{WJ}, \cite{Kru15}).

The aim of this paper is to clarify the implications and relationships between these properties of a field $K$,
more precisely:
\begin{enumerate}
\item[\bf(h)] $K$  is henselian (i.e., $K$ admits a nontrivial henselian valuation),
\item[\bf(eh)] any $L$ which is ${\mathcal{L}_\mathrm{ring}}$-elementarily equivalent to $K$ is henselian,
\item[\bf($\emptyset$-def)] $K$ admits a $\emptyset$-definable nontrivial henselian valuation, and
\item[\bf (def)] $K$ admits a definable nontrivial henselian valuation.
\end{enumerate}

There are some immediate implications between these properties, as summarised in the following diagram:~
\footnote{Our convention is that such diagrams implicitly include concatenations of arrows, although we do not draw them. 
For example, \autoref{fig:1} implicitly includes the implication 
${\bf(\emptyset\textbf{-def})}\implies{\bf(h)}$.}

\begin{figure}[H]
\begin{center}
\begin{displaymath}
\xymatrix{
{\bf(\emptyset\textbf{-def})}\ar@{=>}[d]\ar@{=>}[r] & {\bf(eh)}\ar@{=>}[d] \\
{\bf(def)}\ar@{=>}[r] & {\bf(h)}
}
\end{displaymath}
\end{center}
\caption{The obvious implications}
\label{fig:1}
\end{figure}
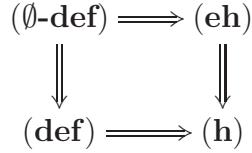

Our aim is to work out the full picture, i.e., to describe which other implications hold,
including which arrows can be reversed.
It turns out that in the class of all fields
(or even in the class $\mathcal{K}_{0}$ of all fields of characteristic zero),
no implications hold that are not already included in \autoref{fig:1} (see part (C) of \autoref{thm:main}).

In order to show this,
we use the canonical henselian valuation $v_K$ to partition $\mathcal{K}_{0}$ into
subclasses, depending on the residue characteristic of $v_{K}$:
$$\mathcal{K}_{0,0} = \{K \textrm{ field}\;|\; \mathrm{char}(K)=\mathrm{char}(Kv_K)=0\}$$
and for any prime $p$
$$\mathcal{K}_{0,p} = \{K \textrm{ field}\;|\; \mathrm{char}(K)=0 \textrm{ and }\mathrm{char}(Kv_K)=p\}.$$
See \autoref{section.Sirince} for the definition of the canonical henselian valuation and a proof that these classes
are closed under $\mathcal{L}_\mathrm{ring}$-elementary equivalence.
We then investigate the corresponding pictures with respect to these subclasses which surprisingly turn out to look rather
different in mixed characteristic and equicharacteristic $0$.
As our main result, we obtain the following
\begin{theorem}\label{thm:main}
\begin{enumerate}
\item[(A)] In the class $\mathcal{K}_{0,0}$ the complete picture is
$$
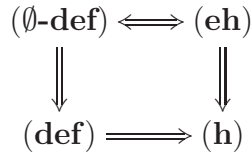
\xymatrix{
{\bf(\emptyset\textbf{\rm\bf-def})}\ar@{=>}[d]\ar@{<=>}[r] & {\bf(eh)}\ar@{=>}[d] \\
{\bf(def)}\ar@{=>}[r] & {\bf(h)}
}$$
\item[(B)] For each prime $p$, in the class $\mathcal{K}_{0,p}$ the complete picture is
$$
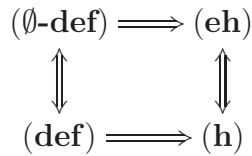
\xymatrix{
{\bf(\emptyset\textbf{\rm\bf-def})}\ar@{<=>}[d]\ar@{=>}[r] & {\bf(eh)}\ar@{<=>}[d] \\
{\bf(def)}\ar@{=>}[r] & {\bf(h)}
}$$
\item[(C)] Consequently, in the class $\mathcal{K}_{0}$ the complete picture is given by \autoref{fig:1}.
\end{enumerate}
\end{theorem}

The paper is organized as follows.
In the next subsection (\autoref{subsection:prelim}), we introduce the basic terminology which we use
throughout the paper 
and discuss the implications and non-implications in our diagrams which are already known. 

In \autoref{section.Sirince},
we recall the definition of the canonical henselian valuation $v_K$ and show that certain properties of the valued field
$(K,v_K)$ are preserved under elementary equivalence in $\mathcal{L}_\mathrm{ring}$ (\autoref{prp:Sirince}). 
In particular,
we obtain that the classes $\mathcal{K}_{0,0}$ and $\mathcal{K}_{0,p}$ (for a fixed prime $p$) are closed under  
$\mathcal{L}_\mathrm{ring}$-elementary equivalence.

In \autoref{section:equicharacteristic.zero}, we show part (A) of \autoref{thm:main}. In order to do this, we first
show the implication which occurs in the picture in (A) but not in \autoref{fig:1} (see \autoref{prp:eh}). 
We then combine this with the examples discussed in \autoref{subsection:prelim} to complete the proof of 
\autoref{thm:main} (A) (see \autoref{proof:A}).

The proof of part \autoref{thm:main} (B) takes some more work. Section \ref{section:divisible-tame} treats the constructions which we use to show the non-implications in the diagram: The main result
of this section is the existence of non-henselian $t$-henselian fields $K$ of any 
characteristic such that there is some tame henselian
$L \equiv K$ with divisible value group (see \autoref{prp:near.hens.divisible.tame}). 

In \autoref{sec:mc.def}, we use the fields constructed in \autoref{section:divisible-tame} and the machinery
developed in \autoref{section:self-similar} to show that for every prime $p$,
there are fields in $\mathcal{K}_{0,p}$ which do not admit $\emptyset$-definable nontrivial henselian valuations (see 
\autoref{ex:mc.no.def}). We then go on to show that for every prime $p$ and every $K \in \mathcal{K}_{0,p}$,
the properties ({\bf def}) and  ($\emptyset$-{\bf def}) are equivalent (see \autoref{thm:mix.char}). Finally, we assemble the facts
we have shown about fields in $\mathcal{K}_{0,p}$ to prove \autoref{thm:main} (B) in \autoref{proof:B}.

\subsection{Preliminaries and known results}\label{subsection:prelim} \label{subsection:JK}
For basic definitions and notions regarding valuation theory, we refer the reader to \cite{EP05}.
We use the following notation: If $(K,v)$ is a valued field, we let
$\mathcal{O}_{v}$ denote the valuation ring,
$\mathfrak{m}_{v}$ denote the maximal ideal,
$Kv$ denote the residue field,
and $vK$ denote the value group.

We call a property of fields is {\em $\mathcal{L}_{\mathrm{ring}}$-elementary} if it is preserved under $\mathcal{L}_{\mathrm{ring}}$-elementary equivalence.
We do not require that the class of fields satisfying that property is an $\mathcal{L}_{\mathrm{ring}}$-elementary class.
The properties {\bf(eh)} and {\bf ($\emptyset$-def)} are obviously $\mathcal{L}_{\mathrm{ring}}$-elementary. 
The existence of non-henselian $t$-henselian fields (first shown by Prestel and Ziegler in \cite[p.\,338]{Prestel-Ziegler78}) shows that
\begin{enumerate}
\item {\bf(h)} is not
$\mathcal{L}_{\mathrm{ring}}$-elementary and
\item ${\bf(h)}$ does not imply ${\bf(eh)}$ for fields in $\mathcal{K}_{0,0}$.
\end{enumerate}
Consequently, {\bf(h)} does not imply {\bf(eh)} for all fields in $\mathcal{K}_{0}$.
The recent paper of Jahnke and Koenigsmann
(\cite{Jahnke-Koenigsmann15})
includes two key examples which are the starting point of our investigation, namely

\begin{example}[Example 6.2, \cite{Jahnke-Koenigsmann15}]
\label{eg:1}
This is an example of a henselian field which does not admit a nontrivial definable henselian valuation. 
In fact, the field $K$ constructed in this example is in the class $\mathcal{K}_{0,0}$ and is 
$\mathcal{L}_\mathrm{ring}$-elementarily equivalent to some
non-henselian field $L$. 
\end{example}

\begin{example}[Example 6.3, \cite{Jahnke-Koenigsmann15}]
\label{eg:2}
This is an example of a henselian field which does admit a nontrivial definable henselian valuation but does not admit a nontrivial $\emptyset$-definable henselian valuation. In fact, the field $K$ constructed is again in the class $\mathcal{K}_{0,0}$ and
 $\mathcal{L}_\mathrm{ring}$-elementarily equivalent to some
non-henselian field $L$. 
\end{example}
\noindent Thus, we get
\begin{enumerate}
\setcounter{enumi}{2}
\item ${\bf(def)}$ is not
$\mathcal{L}_{\mathrm{ring}}$-elementary,
\item ${\bf(h)}$ does not imply ${\bf(def)}$ in $\mathcal{K}_{0,0}$ (and hence in $\mathcal{K}_0$), and
\item ${\bf(def)}$ does not imply ${\bf(\emptyset\textbf{-def})}$  in $\mathcal{K}_{0,0}$ (and hence in $\mathcal{K}_0$).
\end{enumerate}
However, even in the equicharacteristic zero setting there are unanswered questions. 
Perhaps the most obvious is the following, which is labelled `Question 5.6' in \cite{Jahnke-Koenigsmann15}.
\begin{question}
Does ${\bf(eh)}$ imply ${\bf(\emptyset\textbf{-def})}$?
\end{question}
We answer this question negatively for the class of all fields $\mathcal{K}$, however, we show that it does hold
when we restrict our attention to $\mathcal{K}_{0,0}$ (see \autoref{prp:eh}).

\section{The canonical henselian valuation}
\label{section.Sirince}

Recall that any henselian field $K$ may admit many
non-trivial henselian valuations.
However, unless $K$ is separably closed, 
these all induce the same topology on $K$. 
This fact ensures that there is always a canonical one among the henselian valuations on a field.
The \emph{canonical henselian valuation} $v_K$ on $K$
is defined as follows: We divide the class of henselian valuations on $K$ into
subclasses, namely
$$H_1(K) = \{ v \textrm{ henselian on }K\;|\; Kv\textrm{ not separably closed} \}$$
and
$$H_2(K) = \{ v \textrm{ henselian on }K\;|\; Kv\textrm{ separably closed} \}$$
If $H_2(K) \neq \emptyset$, i.e., if $K$ admits a henselian valuation 
with separably closed residue field, then $v_K$ is the (unique) coarsest such. In particular, we have
$v_K \in H_2(K)$. In this case, any
henselian valuation with non-separably closed residue field is a proper coarsening
of $v_K$ and any henselian valuation with separably closed residue field
is a refinement of $v_K$.

If  $H_2(K)= \emptyset$, i.e., if $K$ admits no henselian valuations with separably closed residue field, then $v_K$
is the (unique) finest henselian valuation on $K$ and any two henselian valuations
on $K$ are comparable. In this case, we have $v_K \in H_1(K)$.

Note that whenever $K$ admits some nontrivial henselian valuation then $v_K$ is nontrivial, i.e.,
we have $\mathcal{O}_{v_K}\subsetneq K$.
See \cite[\S 4.4]{EP05} for more details and proofs.

We now show that certain key properties of the valued field $(K,v_{K})$ are in fact $\mathcal{L}_{\mathrm{ring}}$-elementary properties of $K$.

\begin{proposition}\label{prp:Sirince}
The following properties of a field $K$ are $\mathcal{L}_{\mathrm{ring}}$-elementary:
\begin{enumerate}
\item `$v_{K}\in H_{2}(K)$',
\item `$v_{K}$ has residue characteristic $p$',
\item `$v_{K}$ has residue characteristic zero', and
\item `$K$ admits a henselian valuation of mixed characteristic $(0,p)$';
\end{enumerate}
for any given prime $p$.
\end{proposition}
\begin{proof}
Let $L\equiv K$ be a pair of elementarily equivalent fields. In each case we suppose that the relevant property holds in $K$ and show that it also holds in $L$.

\begin{enumerate}
\item Assume that $v_{K}\in H_{2}(K)$. By compactness, there exists an elementary extension $(K,v_{K})\preceq(K^{*},v_{K}^{*})$ such that $L$ elementarily embeds into $K^{*}$; we identify $L$ with its image under this elementary embedding. Let $w$ denote the restriction of $v_{K}^{*}$ to $L$. Since $L$ is relatively algebraically closed in $K^{*}$, $(L,w)$ is henselian. By Hensel's Lemma, $Lw$ is relatively separably algebraically closed in $K^{*}v_{K}^{*}$, and the latter is separably closed. Thus $Lw$ is separably closed. Therefore we get $w\in H_{2}(L)$ and hence $H_2(L)\neq \emptyset$. We conclude $v_L \in H_2(L)$.
\end{enumerate}

Both parts (2) and (3) follow from the following claim.

\begin{claim}
If $K\equiv L$, then the residue characteristics of $v_{K}$ and $v_{L}$ are equal.
\end{claim}
\begin{proof}[Proof of claim]
We will distinguish two cases, based on whether or not $H_{2}(K)$ is empty. By part (1), $H_{2}(K)$ is empty if and only 
if $H_{2}(L)$ is empty. In each case we will use again the construction from part (1) in which we identify $L$ with an elementary subfield of $K^{*}$, where $(K^{*},v_{K}^{*})$ is an elementary extension of $(K,v_{K})$. We let $w$ denote the restriction of $v_{K}^{*}$ to $L$. Since $L$ is relatively algebraically closed in $K^{*}$, $w$ is henselian; and thus $Lw$ is relatively separably closed in $K^{*}v_{K}^{*}$.
\begin{enumerate}
\item[(i)] First we suppose that $H_{2}(K)=\emptyset$. It suffices to show that if one of $(K,v_{K})$ and $(L,v_{L})$ has residue characteristic $p$, then so has the other. Without loss of generality, we suppose that $\mathrm{char}(Kv_{K})=p$. Then $\mathrm{char}(K^{*}v_{K}^{*})=p$; and since $w$ is a restriction of $v_{K}^{*}$, we have that $\mathrm{char}(Lw)=p$.
As $H_{2}(L)=\emptyset$ holds, $v_{L}$ is a (possibly improper) refinement of $w$. Thus $\mathrm{char}(Lv_{L})=p$, as required.

\item[(ii)] Next we suppose that $H_{2}(K)\neq\emptyset$. We first show that if one of $(K,v_{K})$ and $(L,v_{L})$ has residue characteristic zero, then so has the other. Without loss of generality, we suppose that $\mathrm{char}(Kv_{K})=0$. Then $\mathrm{char}(Lw)=0$. Since $v_{K}\in H_{2}(K)$, $Kv_{K}$ and $K^{*}v_{K}^{*}$ are separably closed fields. Since $Lw$ is relatively separably closed in $K^{*}v_{K}^{*}$, $Lw$ is also separably closed. Thus $w$ is a (possibly improper) refinement of $v_{L}$. Thus $\mathrm{char}(Lv_{L})=0$, as required.

Now, assume $\mathrm{char}(Kv_K) =p >0$. In particular, for any $w$ henselian on $K$ we have $\mathrm{char}(Kw)\in \{0,p\}$. 
Take any elementary extension $M$ of $K$. 
Then, we have $\mathrm{char}(Mv_M)>0$ by the above, and the restriction of $v_M$ to $K$ 
is a henselian valuation of mixed characteristic. We conclude $\mathrm{char}(Kv_K)= \mathrm{char}(Mv_M)$.
For any $L\equiv K$ there is some $M$ such that both $K$ and $L$ embed elementarily into $M$. Thus, we get
$\mathrm{char}(Kv_K)=\mathrm{char}(Lv_L)$.
\end{enumerate}
This completes the proof of the claim.
\end{proof}

\begin{enumerate}
\setcounter{enumi}{3}
\item Suppose that $K$ admits a henselian valuation $v$ of mixed characteristic $(0,p)$, for a prime $p$. If $v_{K}$ is of mixed characteristic $(0,p)$ then we simply apply part (3). Otherwise $v_{K}$ is of residue characteristic zero and $v$ is a proper refinement of $v_{K}$. Thus $v_{K}\in H_{2}(K)$, and both $Kv_{K}$ and $Kv$ are separably closed fields. By parts (1) and (2), $Lv_{L}$ is also a separably closed field of characteristic zero. Such fields always carry nontrivial mixed-characteristic henselian valuations.\qedhere
\end{enumerate}
\end{proof}

\begin{corollary}\label{cor:mc} Let $K$ be a non-separably closed field.
The property
\begin{enumerate}
\item[\bf(mc)] `$K$ admits some mixed characteristic henselian valuation'
\end{enumerate}
implies that $K$ is elementarily henselian.
\end{corollary}
\begin{proof}
By part (4) of Proposition \ref{prp:Sirince}, all fields $L$ elementarily equivalent to $K$ admit mixed characteristic henselian valuations. Such valuations are necessarily nontrivial. Thus $L$ is henselian.
\end{proof}

As the contrapositive of Corollary \ref{cor:mc}, we obtain: if $K$ is a non-separably closed non-elementarily henselian field 
then all henselian valuations on fields $L\equiv K$ are equicharacteristic and $H_{2}(L)=\emptyset$.

\section{Fields of equicharacteristic zero}
\label{section:equicharacteristic.zero}

In this section, we show part (A) of \autoref{thm:main}. Note that we only 
need to show one further arrow to complete the picture, namely
({\bf eh}) $\implies$ ($\emptyset$-{\bf def}). This is done in \autoref{section:eh.implies.0-def}. Afterwards,
in \autoref{proof:A}, we explain why combined with the results in \autoref{subsection:prelim}, this indeed proves \autoref{thm:main} part (A).
\subsection{`Elementarily henselian' implies `$\emptyset$-definable'}
\label{section:eh.implies.0-def}
In this subsection, we show why in the class $\mathcal{K}_{0,0}$ of fields $K$ with $\mathrm{char}(Kv_K)=0$, the
implication ({\bf eh}) $\implies$ ($\emptyset$-{\bf def}) holds.
We will apply the following theorem from \cite{Jahnke-Koenigsmann15}.

\begin{theorem}[Theorem B, \cite{Jahnke-Koenigsmann15}]\label{thm:JK.B}
Let $K$ be a non-separably closed henselian field. 
Then $K$ admits a definable nontrivial henselian valuation (using at most $1$ parameter) unless
\begin{enumerate}
\item $Kv_{K}\neq Kv_{K}^{\mathrm{sep}}$, and 
\item $Kv_{K}\preceq L$ for some henselian $L$ with $v_{L}L$ divisible, and
\item $v_{K}K$ is divisible.
\end{enumerate}
\end{theorem}

\begin{lemma}\label{lem:eh.1}
If $K\in\mathcal{K}_{0,0}$ is elementarily henselian then $K$ admits a nontrivial henselian valuation which is definable using at most $1$ parameter.
In particular, for fields of equi\-characteristic zero, ${\bf(eh)}$ implies ${\bf(def)}$.
\end{lemma}
\begin{proof}
We show the contrapositive.
Let $K\in\mathcal{K}_{0,0}$ and suppose that $K$ does not admit a nontrivial henselian valuation which is definable using at most $1$ parameter.
If $K$ is not henselian then we are done;
otherwise $K$ is henselian and we may apply \autoref{thm:JK.B}.
Therefore:
\begin{enumerate}
\item $Kv_{K}\neq Kv_{K}^{\mathrm{sep}}$, and
\item $Kv_{K}\preceq L$ for some henselian $L$ with $v_{L}L$ divisible, and
\item $v_{K}K$ is divisible.
\end{enumerate}
Both $(K,v_{K})$ and $(L,v_{L})$ are henselian valued fields with divisible value groups.
By applying the Ax--Kochen/Ersov principle (\cite[Theorem 4.6.4]{PD}) several times, we conclude:
\begin{align*}
K&\equiv Kv_{K}((\mathbb{Q}))\\
&\equiv L((\mathbb{Q}))\\
&\equiv Lv_{L}((\mathbb{Q}))((\mathbb{Q}))\\
&\equiv L.
\end{align*}
where $\equiv$ is always meant as elementary equivalence in $\mathcal{L}_\mathrm{ring}$.
Therefore $K\equiv L\equiv Kv_{K}$.
Finally, (1) implies that $v_{K}\notin H_{2}(K)$, so $Kv_{K}$ is not henselian.
Thus $K$ is not elementarily henselian.
\end{proof}

We now want to use \autoref{lem:eh.1} to show our missing arrow. The argument works via the Omitting Types Theorem.
Thus, we first start by giving names to the relevant (partial) types.
\begin{definition}
Let $\phi(x;y)$ be an $\mathcal{L}_{\mathrm{ring}}$-formula, where $x$ and $y$ are single variables, and let $n\in\mathbb{N}$.
Let $\delta_{\phi,n}({y})$ be the $\mathcal{L}_{\mathrm{ring}}$-formula that defines the set of elements ${b}$ such that $\phi(x;{b})$ defines a nontrivial 
$n_{\leq}$-henselian valuation ring. We let $D_{\phi}({y})$ denote the partial type
$$\{\delta_{\phi,n}({y})\;|\;n<\omega\}.$$
\end{definition}

Note that $D_{\phi}(y)$ is realised in $K$ if and only if there exists some $b\in K$ such that $\phi(K;b)$ is a nontrivial 
henselian valuation ring of $K$.

\begin{proposition}\label{prp:eh}
If $K\in\mathcal{K}_{0,0}$ is elementarily henselian then $K$ admits a nontrivial $\emptyset$-definable henselian valuation.
Equivalently, in equicharacteristic zero, we have $${\bf(eh)} \implies {(\emptyset\textrm{-{\bf def}})}.$$
\end{proposition}
\begin{proof}
First we show that there is a single formula which defines (with parameters) a nontrivial henselian valuation ring in every $L\equiv K$.

Consider the following countable set of partial types (with respect to the theory of $K$):
$$\mathcal{D}:=\left\{D_{\phi}({y})\;|\;\phi\in\mathcal{L}_{\mathrm{ring}},\,D_{\phi}(y)\text{ is consistent with }\mathrm{Th}(K)\right\}.$$
We suppose, seeking a contradiction, that none of these types is principal.
By the Omitting Types Theorem (see \cite[Corollary 10.3]{TZ}), there exists
some $L\equiv K$ in which none of these types is realised.
That is: $L$ does not admit a nontrivial definable henselian valuation.
Now \autoref{lem:eh.1} implies that $L$ is not elementarily henselian, 
which contradicts our assumption that $K\equiv L$ is elementarily henselian.

Thus there exists an $\mathcal{L}_{\mathrm{ring}}$-formula $\phi(x;y)$ such that $D_{\phi}(y)$ is principal.
Let $\psi(y)$ be a formula which is consistent and isolates $D_{\phi}(y)$, i.e.
$$K\models\;\forall y\;\big(\psi(y)\longrightarrow\delta_{\phi,n}(y)\big),$$
for all $n<\omega$.
Then $\psi(y)$ defines a nonempty set of realisations of $D_{\phi}(y)$ in any $L\equiv K$.
Each element $a$ in this definable set, together with the formula $\phi(x,y)$, defines a nontrivial henselian valuation;
that is, we have a $\emptyset$-definable family of nontrivial henselian valuations.
It remains to show that we can $\emptyset$-define {\em one such}. 

If $H_{2}(K)\neq\emptyset$ then there exists a nontrivial $\emptyset$-definable henselian valuation, by 
\cite[Theorem A]{Jahnke-Koenigsmann15}.
On the other hand, if $H_{2}(K)=\emptyset$, then all henselian valuations on $K$ are comparable.
Let $\Phi(x)$ be the formula
$$\forall{y}\;\left(\psi({y})\longrightarrow\phi(x;{y})\right).$$
This formula $\emptyset$-defines the intersection of the $\emptyset$-definable family of nontrivial henselian valuation rings shown to exist above; and 
any such intersection is also a nontrivial henselian valuation ring.
\end{proof}

\subsection{The full picture in equicharacteristic zero}

We are now in a position to give the following:

\begin{proof}[Proof of part (A) of \autoref{thm:main}]
\label{proof:A}
Our aim is to establish that the complete picture of implications in the class $\mathcal{K}_{0,0}$ is given by the following diagram.
$$\xymatrix{
{\bf(\emptyset\textbf{\rm\bf-def})}\ar@{=>}[d]\ar@{<=>}[r] & {\bf(eh)}\ar@{=>}[d] \\
{\bf(def)}\ar@{=>}[r] & {\bf(h)}
}$$
The implication
${\bf(eh)}\Longrightarrow{\bf(\emptyset\textbf{\rm\bf-def})}$ was shown in \autoref{prp:eh}.
The other implications in the above diagram already hold in the class of all fields (see \autoref{fig:1}).
Finally \autoref{eg:1} and \autoref{eg:2} 
show that implications that are {\em not} contained in the above diagram do not hold in the class $\mathcal{K}_{0,0}$.
\end{proof}


\section{Fields of divisible-tame type}
\label{section:divisible-tame}

The aim of this section is to show the existence of non-henselian, t-henselian fields in any given characteristic,
which are of divisible-tame type (see \autoref{def:divisible.tame}).
Later, specifically in \autoref{lem:self.similar.eq.char}, we will rely on the existence of such fields.

\begin{definition}\label{def:tame}
A valued field $(K,v)$ of residue characteristic $p$ is {\em tame} if the residue field $Kv$ is perfect, the value group is $p$-divisible, and $(K,v)$ is defectless, i.e. the equation
$$[L:K]=(wL:vK)\cdot[Lw:Kv]$$
holds for every finite extension $(L,w)/(K,v)$.
\end{definition}

For more detail on tame valued fields we refer the reader to \cite{Kuhlmann}.

\begin{definition}\label{def:divisible.tame}
We say that a t-henselian field $k$ is of {\em divisible-tame type} if there exists some $K\equiv k$ and a nontrivial valuation $v$ on $K$ such that $(K,v)$ is tame and $vK$ is divisible.
\end{definition}

Our construction is a slight modification of that found in the recent paper \cite{Fehm-Jahnke15}.
Let $\mathbb{P}$ denote the set of prime numbers. The relevant statement from \cite{Fehm-Jahnke15} is the following.

\begin{lemma}[Lemma 6.4, \cite{Fehm-Jahnke15}]
Let $K_{0}$ be a field of characteristic zero that contains all roots of unity. Let $n\in\mathbb{N}$, $n<q\in\mathbb{P}$ and $P\subseteq\mathbb{P}$. Then there exists a valued field $(K_{1},v)$ with the following properties:
\begin{enumerate}
\item $K_{1}v=K_{0}$ and $vK_{1}=\mathbb{Z}[\frac{1}{p}:p\in\mathbb{P}\setminus P]$
\item $v$ is $n_{\leq}$-henselian but not $q$-henselian
\item $G_{K_{1}}=\langle H_{1},H_{2}\rangle$, where $H_{1}\cong\mathbb{Z}_{q}$ and there is $N\lhd H_{2}$ closed with $N\cong\prod_{p\in P}\mathbb{Z}_{p}$ and $H_{2}/N\cong G_{K_{0}}$.
\end{enumerate}
\end{lemma}

See \cite[Definition 6.1]{Fehm-Jahnke15} for the definition of $n_{\leq}$-henselian.

We need to rewrite this lemma and add to its proof in order to apply it to certain fields of positive characteristic.
We say that a field is {\em $p$-closed} if it does not admit any separable extensions of degree $p$.

\begin{lemma}\label{lem:n.hens.not.q.hens}
Let $p\in\mathbb{P}\cup\{1\}$ and let $K$ be a perfect and $p$-closed field of characteristic exponent $p$ that contains all roots of unity.
Let $n\in\mathbb{N}$, let $q\in\mathbb{P}$, and let $P\subseteq\mathbb{P}$ be such that $p\notin P$ and $n<q$.
Then there exists a valued field $(K',v)$ such that
\begin{enumerate}
\item $K'v=K$,
\item $vK'=\mathbb{Z}\big[\frac{1}{l}\;:\;l\in\mathbb{P}\setminus P\big]$,
\item $(K',v)$ is $n_{\leq}$-henselian,
\item $(K',v)$ is not $q$-henselian,
\item $K^{'}$ is perfect and $p$-closed, and
\item $(K^{'h},v)$ is tame.
\end{enumerate}
\end{lemma}

\begin{proof}
Let $\Gamma:=
\mathbb{Z}\big[\frac{1}{l}\;:\;l\in\mathbb{P}\setminus P\big]$.
We work inside the field $K((x^{\Gamma}))$ of generalized power series, 
together with the $x$-adic valuation which we denote by $v_{x}$.
In fact, $v_{x}$ will also denote the restriction of the $x$-adic valuation to any subfield of $K((x^{\Gamma}))$.
Let $F_{0}:=K(x)$ and let $F:=K(x^{\Gamma})$.

Since $(K((x^{\Gamma})),v_{x})$ is maximal,
$\Gamma$ is $p$-divisible, 
and $K$ is perfect;
$(K((x^{\Gamma})),v_{x})$ is tame.
Let $F^{\mathrm{ra}}:=F^{\mathrm{alg}}\cap K((x^{\Gamma}))$ denote the relative algebraic closure of $F$ in $K((x^{\Gamma}))$
and consider the extension
$$(F^{\mathrm{ra}},v_{x})\subseteq(K((x^{\Gamma})),v_{x}).$$
The residue field extension is trivial, since both residue fields are equal to $K$.
In particular the extension of residue fields is algebraic.
Thus we may apply \cite[Lemma 3.7]{Kuhlmann} to find that $(F^{\mathrm{ra}},v_{x})$ is tame.

Just as in the proof of \cite[Lemma 6.4]{Fehm-Jahnke15}, there is a procyclic subgroup $G_{q}\leq G_{F}$, $G_{q}\cong\mathbb{Z}_{q}$. Let $K'$ be the intersection $E\cap F^{\mathrm{ra}}$ where $E$ is the fixed field of $G_{q}$.
Note that $K'v_{x}=K$ and $vK'=\Gamma$.
The claims that $(K',v)$ is $n_{\leq}$-henselian and not $q$-henselian can be read verbatim from the proof of \cite[Lemma 6.4]{Fehm-Jahnke15}.

Note that $F$ is perfect, and so are all algebraic extensions of $F$. Above we showed that $F^{\mathrm{ra}}$ is tame and we assumed that $vF^{\mathrm{ra}}=\Gamma$ is $p$-divisible and that $F^{\mathrm{ra}}v=K$ is $p$-closed; this shows that $p$ does not divide the order of $G_{F^{\mathrm{ra}}}$ (in the supernatural sense).
Since $p\neq q$, neither does $p$ divide the order of $G_{E}\cong\mathbb{Z}_{q}$.
Since
$G_{K'}=\langle G_{E},G_{F^{\mathrm{ra}}}\rangle$,
then $p$ does not divide the order of $G_{K'}$.
Consequently both $K'$ and $K^{'h}$ are $p$-closed; and $(K^{'h},v)$ is tame, as required.
\qedhere
\end{proof}

\begin{lemma}\label{lem:projective.limit}
Let $K$ be a perfect field equipped with a family of equicharacteristic valuations $(v_{n})_{n<\omega}$ such that the corresponding valuation rings $(\mathcal{O}_{n})_{n<\omega}$ form an increasing chain.
Suppose that
\begin{enumerate}
\item $K=\bigcup_{n<\omega}\mathcal{O}_{n}$,
\item $v_{0}K$ is divisible, and
\item $(Kv_{n},\overline{v_{0}})$ is defectless, for each $n<\omega$;
\end{enumerate}
where $\overline{v_{0}}$ denotes the valuation induced on $Kv_{n}$ by $v_{0}$ for $n<\omega$.
Then $(K,v_{0})$ is defectless.
\end{lemma}
\begin{proof}
First note that all value groups that appear in this proof are divisible, since $v_{0}K$ is divisible; and all fields are perfect, since $K$ is perfect.

Let $(K^{h},v_{0})$ denote the henselisation of $(K,v_{0})$.
As a small abuse of notation,
let $v_{n}$ (respectively, $\mathcal{O}_{n}$) also denote the
unique extension to $K^{h}$ of the valuation (resp., valuation ring) of the same name, for each $n<\omega$.
Since $K^{h}/K$ is algebraic, there is no nontrivial valuation on $K^{h}$ which is coarser than all of the valuations $v_{n}$.
Therefore $K^{h}=\bigcup_{n<\omega}\mathcal{O}_{n}$.

We will show that $(K^{h},v_{0})$ is tame.
Since it is a perfect equicharacteristic valued field,
it suffices to show that $(K^{h},v_{0})$ is algebraically maximal, by \cite[Corollary 3.4a]{Kuhlmann}.
Let $(L,w_{0})/(K^{h},v_{0})$ be a finite immediate extension,
i.e. $w_{0}L=v_{0}K^{h}$ and $Lw_{0}=K^{h}v_{0}$.
Our aim is to show that this extension is trivial.
Let $w_{n}$ denote the unique extension to $L$ of the valuation $v_{n}$ and let $\mathcal{O}_{w_{n}}$ (respectively, $\mathfrak{m}_{w_{n}}$) be the valuation ring (resp., maximal ideal) of $w_{n}$.
Just as argued above for $K^{h}$, we note that
$L=\bigcup_{n<\omega}\mathcal{O}_{w_{n}}$,
and equivalently
$\{0\}=\bigcap_{n<\omega}\mathfrak{m}_{w_{n}}$.

We may assume that $L=K^{h}(\alpha)$,
for some $\alpha\in\mathcal{O}_{w_{0}}$.
Let $f\in\mathcal{O}_{0}[x]$ be the minimal polynomial of $\alpha$ over $K^{h}$
(note that $\mathcal{O}_{w_{0}}$ is integral over $\mathcal{O}_{0}$).
Since $K^{h}$ is perfect, $f$ is separable.
Thus $f(\alpha)=0$ and $Df(\alpha)\neq0$, where $Df$ denotes the formal derivative of $f$.
We choose $n<\omega$ such that
$Df(\alpha)\notin\mathfrak{m}_{w_{n}}$.
Writing this in another way, we have
$$(Dfv_{n})(\alpha w_{n})=Df(\alpha)w_{n}\neq0.$$
Trivially we have $(fv_{n})(\alpha w_{n})=f(\alpha)w_{n}=0$.
Thus $\alpha w_{n}\in Lw_{n}$ is a simple root of $fv_{n}$.

Since $(K^{h}v_{n},\overline{v_{0}})$ is defectless,
$(Lw_{n},\overline{w_{0}})/(K^{h}v_{n},\overline{v_{0}})$ is a defectless extension.
Both value groups $\overline{w_{0}}(Lw_{n})$ and $\overline{v_{0}}(K^{h}v_{n})$ are convex subgroups of divisible groups; thus they are divisible.
Since the extension is finite, the extension of value groups is trivial,
Therefore
$$[Lw_{n}:K^{h}v_{n}]=[(Lw_{n})\overline{w_{0}}:(K^{h}v_{n})\overline{v_{0}}]=[Lw_{0}:K^{h}v_{0}].$$
Since we assumed that $(L,w_{0})/(K^{h},v_{0})$ is immediate, $Lw_{0}=K^{h}v_{0}$.
Therefore $Lw_{n}=K^{h}v_{n}$.

Putting all of this together, $\alpha w_{n}\in K^{h}v_{n}$ is a simple root of $fv_{n}$.
Since $(K^{h},v_{n})$ is henselian,
there exists $a\in\mathcal{O}_{n}\subseteq K^{h}$ such that $av_{n}=\alpha w_{n}$ and $f(a)=0$.
Thus $L=K^{h}$.
This shows that $(K^{h},v_{0})$ does not have any proper immediate algebraic extensions, as required.
\end{proof}

\begin{proposition}\label{prp:near.hens.divisible.tame}
Let $p$ be a prime or zero. There exists a non-henselian t-henselian field of characteristic $p$ of divisible-tame type.
\end{proposition}

This proposition is our version of \cite[Construction 6.5]{Fehm-Jahnke15}, which uses our \autoref{lem:n.hens.not.q.hens} instead of \cite[Lemma 6.4]{Fehm-Jahnke15}.
As such, our proof is very similar to that of \cite[Construction 6.5]{Fehm-Jahnke15}.
Nevertheless, we go into some detail in order to be able to highlight the points of difference.

\begin{proof}
Let $K_{0}$ be any field of characteristic $p$ which is perfect and $p$-closed but not separably closed.
For each $n<\omega$ with $n>p$, we choose a prime $q_{n}$ which is greater than $n$.
We apply \autoref{lem:n.hens.not.q.hens} (always with $P=\emptyset$) to obtain a valued field $(K_{1},v_{1})$ which is $n_{\leq}$-henselian,
not $q_{n}$-henselian,
and defectless.
Also $K_{1}$ is of characteristic $p$ and is perfect and $p$-closed.
Finally, $K_{1}v_{1}=K_{0}$, and the value group $v_{1}K_{1}=\mathbb{Q}$ is divisible
(because $P=\emptyset$).

We continue to apply \autoref{lem:n.hens.not.q.hens} recursively.
In this way we obtain a sequence $(K_{n},v_{n})_{n<\omega}$ of valued fields
with the corresponding places forming a chain:
$$\ldots\dashrightarrow K_{n}\overset{v_{n}}{\dashrightarrow }K_{n-1}\dashrightarrow\ldots\overset{v_{1}}{\dashrightarrow} K_{0}.$$
For $n\geq m$, there is the composition
$v_{n,m}:=v_{n}\circ\ldots\circ v_{m+1}$.
This is a valuation on $K_{n}$ with
residue field
$K_{n}v_{n,m}=K_{m}$
and value group
$v_{n,m}K_{n}\equiv\mathbb{Q}$.
For $n\geq m$, we let $\mathcal{O}_{n,m}$ denote the valuation ring corresponding to $v_{n,m}$.
The residue map
$\mathcal{O}_{n,m}\longrightarrow K_{m}$
restricts to a ring epimorphism
$\pi_{n,m}:\mathcal{O}_{n,0}\longrightarrow\mathcal{O}_{m,0}$.
Then the rings
$(\mathcal{O}_{n,0})_{n<\omega}$
together with the maps
$(\pi_{n,m})_{m\leq n}$
form a projective system.

\begin{figure}[h]
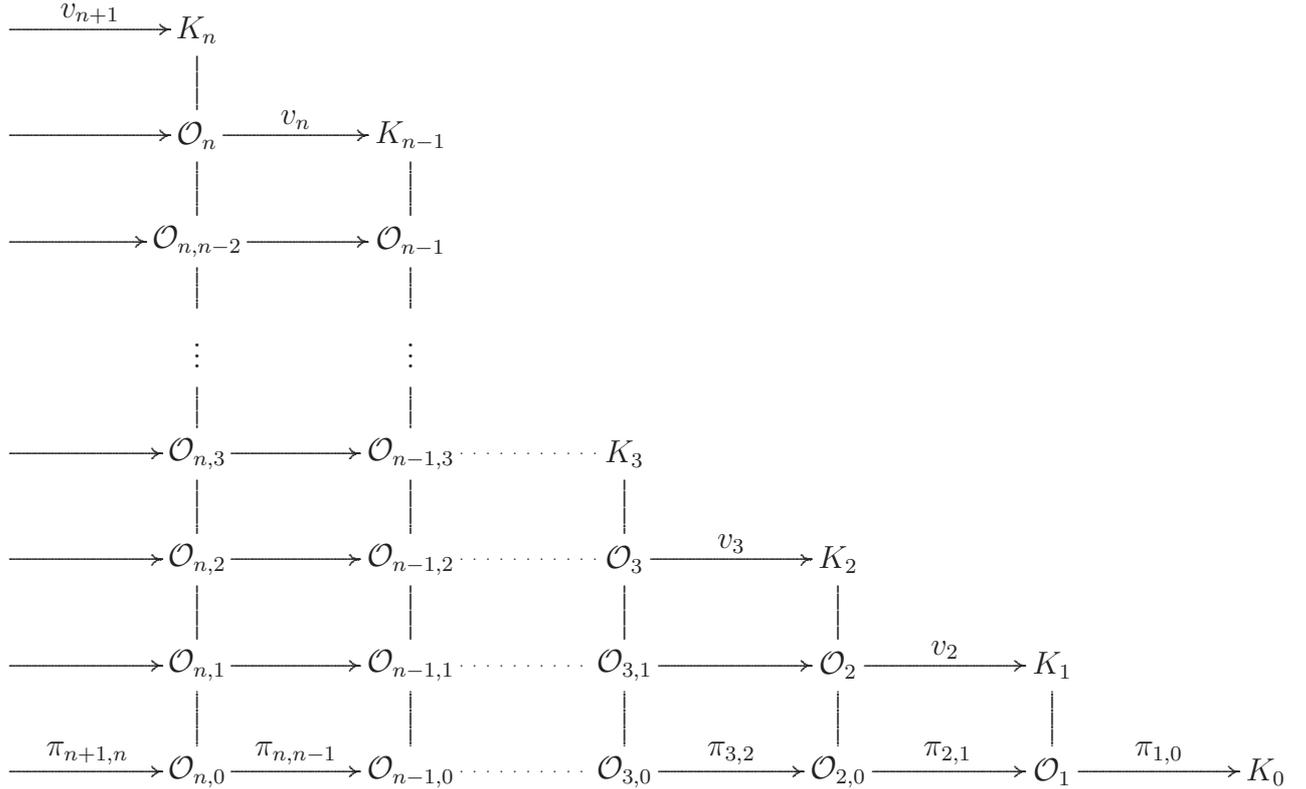

\begin{center}
\begin{displaymath}
\begindc{\commdiag}[40]
\obj(-1,0)[aa]{}
\obj(-1,1)[ab]{}
\obj(-1,2)[ac]{}
\obj(-1,3)[ad]{}
\obj(-1,4)[ae]{}
\obj(-1,5)[af]{}
\obj(-1,6)[ag]{}
\obj(-1,7)[ah]{}
\obj(1,0)[ba]{$\mathcal{O}_{n,0}$}
\obj(1,1)[bb]{$\mathcal{O}_{n,1}$}
\obj(1,2)[bc]{$\mathcal{O}_{n,2}$}
\obj(1,3)[bd]{$\mathcal{O}_{n,3}$}
\obj(1,4)[be]{$\vdots$}
\obj(1,5)[bf]{$\mathcal{O}_{n,n-2}$}
\obj(1,6)[bg]{$\mathcal{O}_{n}$}
\obj(1,7)[bh]{$K_{n}$}
\mor{ba}{bb}{}[\atright,\solidline]
\mor{bb}{bc}{}[\atright,\solidline]
\mor{bc}{bd}{}[\atright,\solidline]
\mor{bd}{be}{}[\atright,\solidline]
\mor{be}{bf}{}[\atright,\solidline]
\mor{bf}{bg}{}[\atright,\solidline]
\mor{bg}{bh}{}[\atright,\solidline]
\mor{aa}{ba}{$\pi_{n+1,n}$}[\atleft,\solidarrow]
\mor{ab}{bb}{}
\mor{ac}{bc}{}
\mor{ad}{bd}{}
\mor{af}{bf}{}
\mor{ag}{bg}{}
\mor{ah}{bh}{$v_{n+1}$}[\atleft,\solidarrow]
\obj(3,0)[ca]{$\mathcal{O}_{n-1,0}$}
\obj(3,1)[cb]{$\mathcal{O}_{n-1,1}$}
\obj(3,2)[cc]{$\mathcal{O}_{n-1,2}$}
\obj(3,3)[cd]{$\mathcal{O}_{n-1,3}$}
\obj(3,4)[ce]{$\vdots$}
\obj(3,5)[cf]{$\mathcal{O}_{n-1}$}
\obj(3,6)[cg]{$K_{n-1}$}
\obj(3,7)[ch]{}
\mor{ca}{cb}{}[\atright,\solidline]
\mor{cb}{cc}{}[\atright,\solidline]
\mor{cc}{cd}{}[\atright,\solidline]
\mor{cd}{ce}{}[\atright,\solidline]
\mor{ce}{cf}{}[\atright,\solidline]
\mor{cf}{cg}{}[\atright,\solidline]
\mor{ba}{ca}{$\pi_{n,n-1}$}[\atleft,\solidarrow]
\mor{bb}{cb}{}
\mor{bc}{cc}{}
\mor{bd}{cd}{}
\mor{bf}{cf}{}
\mor{bg}{cg}{$v_{n}$}[\atleft,\solidarrow]
\obj(5,0)[da]{$\mathcal{O}_{3,0}$}
\obj(5,1)[db]{$\mathcal{O}_{3,1}$}
\obj(5,2)[dc]{$\mathcal{O}_{3}$}
\obj(5,3)[dd]{$K_{3}$}
\mor{da}{db}{}[\atright,\solidline]
\mor{db}{dc}{}[\atright,\solidline]
\mor{dc}{dd}{}[\atright,\solidline]
\mor{ca}{da}{}[\atright,\dotline]
\mor{cb}{db}{}[\atright,\dotline]
\mor{cc}{dc}{}[\atright,\dotline]
\mor{cd}{dd}{}[\atright,\dotline]
\obj(7,0)[ea]{$\mathcal{O}_{2,0}$}
\obj(7,1)[eb]{$\mathcal{O}_{2}$}
\obj(7,2)[ec]{$K_{2}$}
\mor{ea}{eb}{}[\atright,\solidline]
\mor{eb}{ec}{}[\atright,\solidline]
\mor{da}{ea}{$\pi_{3,2}$}[\atleft,\solidarrow]
\mor{db}{eb}{}
\mor{dc}{ec}{$v_{3}$}[\atleft,\solidarrow]
\obj(9,0)[fa]{$\mathcal{O}_{1}$}
\obj(9,1)[fb]{$K_{1}$}
\mor{fa}{fb}{}[\atright,\solidline]
\mor{ea}{fa}{$\pi_{2,1}$}[\atleft,\solidarrow]
\mor{eb}{fb}{$v_{2}$}[\atleft,\solidarrow]
\obj(11,0)[ga]{$K_{0}$}
\mor{fa}{ga}{$\pi_{1,0}$}[\atleft,\solidarrow]
\enddc
\end{displaymath}
\end{center}
\caption{The projective system}
\label{fig:3}
\end{figure}

Let $\mathcal{O}$
together with the natural projections 
$\pi_{\infty,n}:\mathcal{O}\longrightarrow\mathcal{O}_{n,0}$
be the projective limit of this system,
and let $K$ be the quotient field of $\mathcal{O}$.
By \cite[Lemma 3.5]{Fehm-Paran11}
, $\mathcal{O}$ is a valuation ring.

For each $n<\omega$, let $\mathfrak{p}_{n}$ denote the kernel of $\pi_{\infty,n}$,
and let $\mathcal{O}_{\mathfrak{p}_{n}}$ denote the localisation of $\mathcal{O}$ at $\mathfrak{p}_{n}$.
Since
$\mathfrak{p}_{n}\supseteq\mathfrak{p}_{n+1}$,
we have $\mathcal{O}_{\mathfrak{p}_{n}}\subseteq\mathcal{O}_{\mathfrak{p}_{n+1}}$.
Since $\{0\}=\bigcap_{n<\omega}\mathfrak{p}_{n}$, we have $K=\bigcup_{n<\omega}\mathcal{O}_{\mathfrak{p}_{n}}$.

Let $v^{*}_{n}$ denote the valuation on $K$ with valuation ring $\mathcal{O}_{n}^{*}:=\mathcal{O}_{\mathfrak{p}_{n}}$.
Then $(v_{n}^{*})_{n<\omega}$ is a strictly increasing (i.e. increasingly coarse) chain of valuations on $K$;
and the finest common coarsening of this chain is the trivial valuation. 

For each $n<\omega$, $v_{0}^{*}$ induces a valuation $\overline{v_{0}}$ on $Kv_{n}^{*}=K_{n}$.
In fact this valuation is equal to the composition $v_{n,0}$ which was described above.
As the composition of defectless valuations (see
\autoref{lem:n.hens.not.q.hens})
, $\overline{v_{0}}=v_{n,0}$ is defectless.
Furthermore, the value group $\overline{v_{0}}(Kv_{n}^{*})=v_{n,0}K_{n}$ is an extension of divisible groups; thus it is divisible.

We have shown that the hypotheses of \autoref{lem:projective.limit} are satisfied.
Therefore $(K,v_{0}^{*})$ is defectless.

The field $K$ is non-henselian by exactly the same arguments as in \cite[Proposition 6.7]{Fehm-Jahnke15}, which we omit here.

Finally, let $K^{*}$ be an elementary extension of $K$ in which the least common coarsening $v^{*}$ of $v_{n}^{*}$ is nontrivial.
For example, we may take a nonprincipal ultraproduct of the family $(K,v_{n}^{*})$, $n<\omega$.
Then $(K^{*},v^{*})$ is a perfect nontrivially valued field, which is henselian and defectless, and has divisible value group.
Therefore $K$ is of divisible-tame type, as required.
\end{proof}


\section{Fields of mixed-characteristic}

The goal of this section is to prove part (B) of \autoref{thm:main}.
We've already seen in \autoref{cor:mc} that ${\bf(mc)}$ implies  ${\bf(eh)}$.
This leaves us with showing for mixed characteristic fields that
\begin{enumerate}
\item ${\bf(h)}$ does not imply $(\emptyset\text{\rm\bf-def})$, and
\item ${\bf(def)}$ implies $(\emptyset\text{\rm\bf-def})$.
\end{enumerate}

\subsection{Self-similarity}
\label{section:self-similar}

In the following lemma, we adapt the Ax--Kochen argument from \autoref{lem:eh.1} to the slightly more general setting
of $t$-henselian fields of divisibly tame type.

\begin{lemma}\label{lem:self.similar.eq.char}
Let $k$ be an equicharacteristic t-henselian field of divisible-tame type.
Then $k\equiv k((\mathbb{Q}))$.

Thus $k$ is elementarily equivalent to all equicharacteristic nontrivial tame valued fields with residue field elementarily equivalent to $k$ and divisible value group.
\end{lemma}
\begin{proof}
By definition of `divisible-tame type' there exists a field $K\equiv k$ and a nontrivial valuation $v$ on $K$ such that $(K,v)$ is tame and $vK$ is divisible.
By repeated use of the Ax--Kochen principle for equicharacteristic tame valued fields (see \cite[Theorem 1.4]{Kuhlmann}), 
we get as in \autoref{lem:eh.1}
\begin{align*} K &\equiv Kv((\mathbb{Q}))\\
&\equiv Kv((\mathbb{Q}))((\mathbb{Q}))\\
&\equiv k((\mathbb{Q}))
\end{align*}
where $\equiv$ always stands for elementary equivalence in $\mathcal{L}_\mathrm{ring}$.

The second claim follows from the completeness of the theory of valued fields with nontrivial tame valuations, divisible value groups, and residue fields elementarily equivalent to $k$.
\end{proof}

Of course, in mixed-characteristic, a field cannot be elementarily equivalent to its residue field, simply for reasons of characteristic.
Instead, we give the following definition.

\begin{definition}
We say a valued field $(L,w)$ is {\em self-similar} if there is an elementary extension $(L^{*},w^{*})\succeq(L,w)$ and a valuation $u$ on $L^{*}$ which is not equal to $w^{*}$ such that $(L^{*},w^{*})\equiv(L^{*},u)$.
\end{definition}

It is clear that if $(L,w)$ is self-similar then $w$ cannot be $\emptyset$-definable.

\begin{proposition}\label{lem:symmetry.mixed.characteristic}
Let $p$ be any prime.
Let $k$ be a field of characteristic $p$ which is $t$-henselian and of divisible-tame type.
Let $(L,w)$ be a
mixed-characteristic
tame valued field of with $wL\equiv\mathbb{Q}$ and $Lw=k$.
Then $(L,w)$ is self-similar.
\end{proposition}
\begin{proof}
By definition of `divisible-tame type' there exists a field $K\equiv k$ and a nontrivial equicharacteristic valuation $v$ on $K$ such that $(K,v)$ is tame and $vK$ is divisible.
By \autoref{lem:self.similar.eq.char}, $K\equiv Kv$.

By the Keisler-Shelah Theorem (\cite[Theorem 8.5.10]{hod}), we may assume that $K$ is an ultraproduct of $k$.
Let $(L^{*},w^{*})$ be the corresponding ultraproduct of $(L,w)$;
thus $(L^{*},w^{*})\succeq(L,w)$ is a tame valued field of mixed characteristic
with a divisible value group and $L^{*}w^{*}=K$.

Consider the field $K((\mathbb{Q}))$ with the $t$-adic valuation $v_{t}$.
Let $v':=v\circ v_{t}$ denote the composition of $v$ and $v_{t}$.
Then $(K((\mathbb{Q})),v')/(K,v)$ is an extension of tame valued field with divisible value groups.
Note that the theory of divisible ordered abelian groups is model complete.
The extension of residue fields is trivial:
both residue fields are $Kv$.
By the $\mathrm{AKE}_{\preceq}$-principle for equicharacteristic tame valued fields (see \cite[Theorem 1.4]{Kuhlmann}) we have $(K,v)\preceq(K((\mathbb{Q})),v')$.

Let $(L',w')$ denote an extension of $(L^{*},w^{*})$ chosen so that
$w'L'$ is divisible and
$L'w'=K((\mathbb{Q}))$.
By passing to a maximal immediate extension if necessary,
we may assume that $(L',w')$ is tame.

$$
\xymatrix{
L'\ar@{->}[r] & K((\mathbb{Q}))\ar@{->}[r] & K\ar@{->}[r] & Kv\\
L^{*}\ar@{-}[u]\ar@{->}[r] & K\ar@{-}[u]\ar@{->}[rr]\ar@{-}[ur] & & Kv\ar@{-}[u]\\
L\ar@{-}[u]\ar@{->}[r] & k\ar@{-}[u] \\
}
$$

Let $u:=v_{t}\circ w^{'}$ be the composition of $v$ and $w^{*}$.
Then both $(L',w')$ and $(L',u)$ are tame valued fields which extend $(L^{*},w^{*})$.
The residue field extension of $(L',w')/(L^{*},w^{*})$ is $K((\mathbb{Q}))/K$, which is an elementary extension.
The residue field extension of $(L',u)/(L^{*},w^{*})$ is $K/K$, which is trivial, thus elementary.
Again, note also that all extensions of divisible value groups are elementary because the theory of divisible ordered abelian groups is model complete.

By the $\mathrm{AKE}_{\preceq}$-principle for mixed-characteristic tame valued fields (see \cite[Theorem 1.4]{Kuhlmann}),
$(L',w')\equiv_{(L^{*},w^{*})}(L',u)$.
In particular $(L',w')\equiv (L',u)$.
Therefore $(L,w)$ is self-similar, as required.
\end{proof}

\subsection{`Henselian' does not imply `definable'} \label{sec:mc.def}

One of the remaining questions in mixed characteristic (short of giving a characterisation of fields with ${\bf(def})$) is whether or not all 
fields in $\mathcal{K}_{0,p}$ admit definable nontrivial henselian valuations. The answer is `not'. 
For any prime $p$, we exhibit in \autoref{ex:mc.no.def} a field in $\mathcal{K}_{0,p}$ which does not admit a $\emptyset$-definable nontrivial henselian valuation. By 
\autoref{thm:mix.char}, these fields do not even admit a definable nontrivial henselian valuation.

\begin{proposition}
Let $(L,w)$ be a henselian valued field of mixed characteristic such that $wL=\mathbb{Q}$ and $Lw$ is a non-henselian
$t$-henselian field of divisible-tame type. \label{prp:self}
Then $L$ does not admit a $\emptyset$-definable nontrivial henselian valuation.
\end{proposition}
\begin{proof}
First note that $w$ is the only nontrivial henselian valuation on $L$.
By \autoref{lem:symmetry.mixed.characteristic}, there exists $(L,w)\preceq(L^{*},w^{*})$ and a valuation $u$ on $L^{*}$ which is different from $w^{*}$ such that $(L^{*},w^{*})\equiv(L^{*},u)$.
Consequently, $w^{*}$ is not $\emptyset$-definable in $L^{*}$, and $w$ is not $\emptyset$-definable in $L$.
\end{proof}

\begin{example}\label{ex:mc.no.def}
Let $p$ be any prime and let $k$ be a field of characteristic $p$ which is non-henselian but $t$-henselian of divisible-tame type,
e.g., any field constructed in the proof of \autoref{prp:near.hens.divisible.tame}.
Let $(L,w)$ be a mixed-characteristic tame valued field with $wL=\mathbb{Q}$ and $Lw=k$.

By \autoref{prp:self}, $L$ does not admit any $\emptyset$-definable nontrivial henselian valuation. 
\end{example}

\subsection{`Definable' implies `$\emptyset$-definable'}
The aim of this subsection is to show that for any prime $p$ and any $K \in \mathcal{K}_{0,p}$, we have
$$({\textbf {def}}) \implies (\emptyset\textrm{-}{\textbf {def}})$$ 

The proof uses the machinery of $q$-henselian valuations as developed in \cite{Jahnke-Koenigsmann14}.
Let $q$ be any prime. 
Recall that a valuation $v$ on a field $L$ is called $q$-henselian if $v$ extends uniquely to every Galois extension of $L$
of $q$-power degree. Let $L$ be a field
admitting nontrivial Galois extensions of $q$-power degree; we denote this by $L\neq L(q)$. 
Then, there is always a canonical $q$-henselian
valuation $v_L^q$, and the definiton is similar to that of the canonical henselian valuation.
Again, we divide the class of $q$-henselian valuations on $L$
 into
two subclasses, namely
$$H^q_1(L) = \{v\; q\textrm{-henselian on } L \,|\, Lv \textrm{ admits a Galois extension of degree }q\}$$
and
$$H^q_2(L) = \{ v\; q\textrm{-henselian on } L \,|\, Lv \textrm{ does not admit a Galois extension of degree }q \}.$$ 
One can deduce that any valuation $v_2 \in H^q_2(L)$ 
is \emph{finer} than any $v_1 \in H^q_1(L)$, i.e. 
${\mathcal O}_{v_2} \subsetneq {\mathcal O}_{v_1}$,
and that any two valuations in $H^q_1(L)$ are comparable.
Furthermore, if $H^q_2(L)$ is non-empty, then there exists a unique coarsest
valuation
$v_L^q$ in $H^q_2(L)$; otherwise there exists a unique finest 
valuation $v_L^q \in H^q_1(L)$.
In either case, $v_L^q$ is called the \emph{canonical $q$-henselian valuation}.
If $L$ is $q$-henselian then $v_L^q$ is non-trivial.
Note that any henselian valuation on $L$ is $q$-henselian and thus comparable to $v_L^q$

Our proof uses a special case of
the uniform definabilty of canonical $q$-henselian valuation as proven in \cite[Main Theorem]{Jahnke-Koenigsmann14}: 
Let $\mathcal{F}_q$ be the (elementary) class of fields $L$ such that $L$ has characteristic away from
$q$ and admits a Galois extension of degree $q$,
 and such that $L$ contains a primitive $q$th root of unity $\zeta_q$. In case $q=2$, assume further that $L$ is non-orderable.
There is a parameter-free 
$\mathcal{L}_\textrm{ring}$-formula $\varphi(x)$ such that we have
$$L \in \mathcal{F}_q \implies \varphi(L) = \mathcal{O}_{v_L^q}.$$

Furthermore, we will make repeated use of the following
\begin{fact}[{\cite[p.\,43 and Corollary 4.1.4]{EP05}}]\label{fact:overrings.chain}
Let $\mathcal{O}\subseteq K$ be a valuation ring. The overrings of $\mathcal{O}$ in $K$ form a chain under inclusion and each overring is a valuation ring. If $\mathcal{O}$ is henselian, then all overrings of $\mathcal{O}$ in $K$ are henselian.
\end{fact}

We can now prove the main result of this subsection:
\begin{theorem}\label{thm:mix.char}
If $(K, v_K)$ has mixed-characteristic then
$${\bf(def)}\Longrightarrow{\bf(\emptyset\textbf{\rm\bf-def})}.$$
\end{theorem}
\begin{proof} Fix a prime $p$.
Let $K$ be a field with $\mathrm{char}(K)=0$ and $\mathrm{char}(Kv_K)=p>0$ which admits a definable
nontrivial henselian valuation. In particular, $K$ is not separably closed as no separably closed field admits a definable
nontrivial henselian valuation. Furthermore,
 by \cite[Theorem A]{Jahnke-Koenigsmann15}, we may assume that $Kv_{K}\neq Kv_{K}^{\mathrm{sep}}$.
Since $v_{K}$ has mixed characteristic, $v_{K}$ is nontrivial.
Thus there exists a prime $q$ and a finite extension $L_{0}/K$ such that $L_{0}\neq L_{0}(q)$ 
and $\zeta_{q}\in L_{0}$. Let $n:=[L_{0}:K]$ and define
$$\mathcal{L}:=\{L\;|\;[L:K]=n,L\neq L(q),\zeta_{q}\in L\}.$$
The family $\mathcal{L}$ is uniformly interpretable in $K$: we quantify over those $n$-tuples from $K$ which are the coefficients of irreducible polynomials over $K$, such polynomials generate extensions $L/K$ and we can define those tuples of coefficients of polynomials that generate extensions $L\in\mathcal{L}$.

Next we explain a few basic facts about the canonical $q$-henselian 
valuations $v_{L}^{q}$ that we will repeatedly use. Let $L\in\mathcal{L}$. Since $L/K$ is a finite extension and $Kv_{K}$ is not separably closed, $v_{L}$ is the unique extension of $v_{K}$ to $L$. Since $v_{L}^{q}$ is comparable to $v_{L}$, $v_{L}^{q}|_{K}$ is also comparable to $v_{K}$. Again, since $v_{L}^{q}$ is comparable to $v_{L}$, the residue characteristic of $v_{L}^{q}$ is either $0$ or $p$. Finally, since $v_{L}$ is a nontrivial $q$-henselian valuation and $L\neq L(q)$, we have that $v_{L}^{q}$ and thus
$v_{L}^{q}|_{K}$ are nontrivial.

We define
$$\mathcal{L}_{1}:=\{L\in\mathcal{L}\;|\;\mathrm{char}(Lv_{L}^{q})=p\}$$
and
$$\mathcal{L}_{2}:=\{L\in\mathcal{L}\;|\;\mathrm{char}(Lv_{L}^{q})\neq p\}=\{L\in\mathcal{L}\;|\;\mathrm{char}(Lv_{L}^{q})=0\}.$$

Just as for $\mathcal{L}$ above, both $\mathcal{L}_{1}$ and $\mathcal{L}_{2}$ are uniformly interpretable in $K$. To see that $\mathcal{L}_{1}$ is uniformly interpretable: given a uniform interpretation of $\mathcal{L}$, we then need to define which $n$-tuples correspond to extensions $L/K$ such that $\mathrm{char}(Lv_{L}^{q})=p$, and this follows from the fact that $v_{L}^{q}$ is uniformly $\emptyset$-definable in $L$, by \cite[Main Theorem]{Jahnke-Koenigsmann14}. 
Let $\Lambda_{1}(\mathbf{y})$ and $\Lambda_{2}(\mathbf{y})$ be the formulas that define those $n$-tuples corresponding to extensions $L/K$ in $\mathcal{L}_{1}$ and $\mathcal{L}_{2}$, respectively.

We proceed by a case distinction. In each case our goal is of course to find an $\emptyset$-definable nontrivial henselian valuation on $K$.

\bf Case 1\rm: Suppose first that $\mathcal{L}_{2}\neq\emptyset$ and let $L\in\mathcal{L}_{2}$. As noted above, $\mathcal{O}_{L}^{q}$ is comparable to $\mathcal{O}_{L}$. Since $L\in\mathcal{L}_{2}$, $\mathrm{char}(Lv_{L}^{q})=0$. Thus $\mathcal{O}_{L}\subset\mathcal{O}_{L}^{q}$ and $\mathcal{O}_{K}=\mathcal{O}_{L}\cap K\subset\mathcal{O}_{L}^{q}\cap K$. We have the following diagram.
$$
\xymatrix{
& L \\
K\ar@{-}[ru] & \mathcal{O}_{L}^{q}\ar@{-}[u] \\
\mathcal{O}_{L}^{q}\cap K\ar@{-}[u]\ar@{-}[ur] & \mathcal{O}_{L}\ar@{-}[u]\\
\mathcal{O}_{K}\ar@{-}[u]\ar@{-}[ur] & 
}
$$
We let
$$\mathcal{O}_{1}:=\bigcap_{L\in\mathcal{L}_{2}}\mathcal{O}_{L}^{q}\cap K.$$
It is immediate that $\mathcal{O}_{K}\subseteq\mathcal{O}_{1}$. By \autoref{fact:overrings.chain}, $\mathcal{O}_{1}$ is an henselian valuation ring. As noted above, each $\mathcal{O}_{L}^{q}\cap K$ is nontrivial. Since $\mathcal{O}_{1}\subseteq\mathcal{O}_{L}^{q}\cap K$, for each $L\in\mathcal{L}_{2}$, $\mathcal{O}_{1}$ is also nontrivial.

Finally, $\mathcal{O}_{1}$ is $\emptyset$-defined in $K$ by the formula
$$\forall \mathbf{y}\;(\Lambda_{2}(\mathbf{y})\longrightarrow\phi_{q}(x,\mathbf{y})).$$

\bf Case 2\rm: Now suppose that $\mathcal{L}_{2}=\emptyset$. We have not used thusfar that $K$ admits a nontrivial definable henselian valuation. Let $\phi(x,t)$ be an $\mathcal{L}_{\mathrm{ring}}$-formula with parameter $t\in K$ that defines in $K$ a nontrivial henselian valuation ring $\mathcal{O}_{t}$, i.e. $\phi(K,t)=\mathcal{O}_{t}$.

For $L\in\mathcal{L}$, let $\mathcal{O}_{t,L}$ denote the unique extension of $\mathcal{O}_{t}$ to $L$. Then $\mathcal{O}_{t,L}$ is henselian, thus $q$-henselian. Therefore $\mathcal{O}_{t,L}$ is comparable to $\mathcal{O}_{L}^{q}$, and so their restrictions to $K$ (which are $\mathcal{O}_{t}$ and $\mathcal{O}_{L}^{q}\cap K$) are comparable. Therefore
$$\mathcal{L}_{1}=\{L\in\mathcal{L}_{1}\;|\;\mathcal{O}_{L}^{q}\cap K\subseteq\mathcal{O}_{t}\}\sqcup\{L\in\mathcal{L}_{1}\;|\;\mathcal{O}_{t}\subset\mathcal{O}_{L}^{q}\cap K\}.$$
This allows us to distinguish two subcases: in {\bf Case 2a}, for \bf some \rm $L\in\mathcal{L}_{1}$ the ring $\mathcal{O}_{L}^{q}\cap K$ is a strict coarsening of $\mathcal{O}_{t}$; whereas in {\bf Case 2b}, for \bf every \rm $L\in\mathcal{L}_{1}$ the ring $\mathcal{O}_{L}^{q}\cap K$ is a refinement of $\mathcal{O}_{t}$.

In the meantime, we let $S:=\bigcup_{L\in\mathcal{L}_{1}}\mathcal{O}_{L}^{q}\cap K$, and note that $S$ is $\emptyset$-defined in $K$ by the formula
$$\exists \mathbf{y}\;(\Lambda_{1}(\mathbf{y})\wedge\phi_{q}(x,\mathbf{y})).$$
As $S$ is a union of valuation rings each of which is comparable to $\mathcal{O}_{t}$, $S$ is also comparable to $\mathcal{O}_{t}$. In fact, in {\bf Case 2a}, we have $\mathcal{O}_{t}\subset S$; and in {\bf Case 2b}, we have $S\subseteq\mathcal{O}_{t}$.

From now on we separate the subcases.

\bf Case 2a\rm: We suppose that for \bf some \rm $L'\in\mathcal{L}_{1}$ the ring $\mathcal{O}_{L'}^{q}\cap K$ is a strict coarsening of $\mathcal{O}_{t}$. If we let $\mathcal{L}_{1}':=\{L\in\mathcal{L}_{1}\;|\;\mathcal{O}_{t}\subset\mathcal{O}_{L}^{q}\cap K\}$ then our assumption may be rephrased as $\mathcal{L}_{1}'\neq\emptyset$. We will show that $S$ is a mixed characteristic nontrivial henselian valuation ring, and we already know that $S$ is $\emptyset$-definable in $K$. Note that, as discussed above, in this subcase we have $\mathcal{O}_{t}\subset S$, although we do not make direct use of this fact.

For each $L\in\mathcal{L}_{1}\setminus\mathcal{L}_{1}'$, we have
$$\mathcal{O}_{L}^{q}\cap K\subseteq\mathcal{O}_{t}\subset\mathcal{O}_{L'}^{q}\cap K\subseteq S.$$

Consequently
$$S=\bigcup_{L\in\mathcal{L}_{1}}\mathcal{O}_{L}^{q}\cap K=\bigcup_{L\in\mathcal{L}_{1}'}\mathcal{O}_{L}^{q}\cap K,$$
and therefore $S$ is a union of valuation rings each of which is a strict coarsening of $\mathcal{O}_{t}$.
$$
\xymatrix{
K \\
S=\bigcup_{L\in\mathcal{L}_{1}'}\mathcal{O}_{L}^{q}\cap K\ar@{-}[u] \\
\mathcal{O}_{t}\ar@{-}[u] \\
}
$$
By \autoref{fact:overrings.chain}, the coarsenings of $u_{t}$ form a chain under inclusion, and so $S$ is a union of a chain of valuation rings. Therefore $S$ is a valuation ring. Since $S$ coarsens $\mathcal{O}_{t}$, $S$ is henselian. Finally, since $S$ is a union of mixed characteristic valuation rings, $S$ has mixed characteristic. In particular, $S$ is nontrivial.

\bf Case 2b\rm: We suppose that for \bf every \rm $L\in\mathcal{L}_{1}$ the ring $\mathcal{O}_{L}^{q}\cap K$ is a refinement of $\mathcal{O}_{t}$. As noted above, we have $S\subseteq\mathcal{O}_{t}$.

Since $S$ contains a valuation ring (e.g. $\mathcal{O}_{L}^{q}\cap K$, for any $L\in\mathcal{L}_{1}$), the set of subrings of $K$ which contain $S$ is totally ordered, by \autoref{fact:overrings.chain}. Therefore, any (nonempty) union or intersection of rings containing $S$ is also a ring.

Let $u_{t}$ denote the valuation on $K$ corresponding to $\mathcal{O}_{t}$. We now consider a final distinction into (subsub)cases depending on the characteristic of $Ku_{t}$. Note that since $u_{t}$ is henselian, it is a refinement of $v_{K}$ which has mixed characteristic. Thus $\mathrm{char}(Ku_{t})\in\{0,p\}$.

If, for $s\in K$, $\phi(K,s)$ is a valuation ring then it will be denoted $\mathcal{O}_{s}$ and its corresponding valuation will be denoted $u_{s}$.

\bf Case 2b(i)\rm: Suppose that $\mathrm{char}(Ku_{t})=p$. Let
$$\mathcal{O}_{2}:=\bigcup\{\phi(K,s)\;|\;\mathcal{O}_{s}=\phi(K,s)\text{ is val ring},S\subseteq\mathcal{O}_{s},\mathrm{char}(Ku_{s})=p\}.$$
We have the following picture.
$$
\xymatrix{
K \\
\mathcal{O}_{2}\ar@{-}[u] \\
\mathcal{O}_{t}\ar@{-}[u] \\
S\ar@{-}[u] \\
}
$$
As noted above, $\mathcal{O}_{2}$ is a union of a chain of rings containing $S$, thus $\mathcal{O}_{2}$ is a valuation ring in $K$. In fact, since $\mathcal{O}_{2}$ is a union of mixed characteristic valuation rings, $\mathcal{O}_{2}$ has mixed characteristic. Thus $\mathcal{O}_{2}$ is nontrivial.

By \autoref{fact:overrings.chain}, since $\mathcal{O}_{t}=\phi(K,t)\subseteq\mathcal{O}_{2}$, we have that $\mathcal{O}_{2}$ is henselian.

Finally, note that $\mathcal{O}_{2}$ is $\emptyset$-defined in $K$ by the following formula.

$$\exists s\;\Big(\big(V_{\phi}(s)\wedge\forall y\;\big(y\in S\longrightarrow\phi(y,s)\big)\wedge\neg\phi(p^{-1},s)\big)\longrightarrow\phi(x,s)\Big),$$
where, as above, $V_{\phi}(s)$ is a formula defining those $s$ such that $\phi(K,s)$ is a valuation ring. This finishes \bf Case 2b(i)\rm.

\bf Case 2b(ii)\rm: Suppose that $\mathrm{char}(Ku_{t})=0$. Let
$$\mathcal{O}_{3}:=\bigcap\{\phi(K,s)\;|\;\mathcal{O}_{s}=\phi(K,s)\text{ is val ring},S\subseteq\mathcal{O}_{s},\mathrm{char}(Ku_{s})=0\}.$$
We have the following picture.
$$
\xymatrix{
K \\
\mathcal{O}_{t}\ar@{-}[u] \\
\mathcal{O}_{3}\ar@{-}[u] \\
S\ar@{-}[u] \\
}
$$
As noted above, as an intersection of a chain of rings containing $S$, $\mathcal{O}_{3}$ is a valuation ring in $K$. In fact, since $\mathcal{O}_{3}$ is an intersection of equal characteristic valuation rings, $\mathcal{O}_{3}$ has equal characteristic. Since $\mathcal{O}_{3}\subseteq\mathcal{O}_{t}$, $\mathcal{O}_{3}$ is nontrivial.

We claim that $\mathcal{O}_{3}$ is a coarsening of $\mathcal{O}_{K}$, i.e. $\mathcal{O}_{K}\subseteq\mathcal{O}_{3}$. To see this: let $L\in\mathcal{L}$. As noted above, $\mathcal{O}_{L}^{q}\cap K$ is comparable to $\mathcal{O}_{K}$. Either
$$\mathcal{O}_{K}\subseteq\mathcal{O}_{L}^{q}\cap K\subseteq S\subseteq\mathcal{O}_{3},$$
as required; or
$$\mathcal{O}_{L}^{q}\cap K\subset\mathcal{O}_{K}.$$
In the latter case, $\mathcal{O}_{K}$ and $\mathcal{O}_{3}$ are both coarsenings of $\mathcal{O}_{L}^{q}\cap K$; and so they are comparable, by \autoref{fact:overrings.chain}. Since $\mathcal{O}_{3}$ has residue characteristic zero, $\mathcal{O}_{K}\subset\mathcal{O}_{3}$. In either case, we have shown that $\mathcal{O}_{3}$ is a coarsening of $\mathcal{O}_{K}$. Consequently, $\mathcal{O}_{3}$ is henselian.

Finally, note that $\mathcal{O}_{3}$ is $\emptyset$-defined in $K$ by the following formula.

$$\forall s\;\Big(\big(V_{\phi}(s)\wedge\forall y\;\big(y\in S\longrightarrow\phi(y,s)\big)\wedge\phi(p^{-1},s)\big)\longrightarrow\phi(x,s)\Big),$$
where, as above, $V_{\phi}(s)$ is a formula defining those $s$ such that $\phi(K,s)$ is a valuation ring. This finishes \bf Case 2b(ii)\rm.
\end{proof}

\subsection{The full picture in mixed-characteristic}
We can now collect the facts we have proven for fields in $\mathcal{K}_{0,p}$ and assemble them to a proof
of \autoref{thm:main} (B):
\begin{proof}[Proof of part (B) of \autoref{thm:main}]
\label{proof:B}
We want to show that for each prime $p$, in the class $\mathcal{K}_{0,p}$ the complete picture is
$$\xymatrix{
{\bf(\emptyset\textbf{\rm\bf-def})}\ar@{<=>}[d]\ar@{=>}[r] & {\bf(eh)}\ar@{<=>}[d] \\
{\bf(def)}\ar@{=>}[r] & {\bf(h)}
}$$
Apart from the trivial implications as given in \autoref{fig:1},
we have shown in \autoref{cor:mc} that for any $K \in \mathcal{K}_{0,p}$
$$(\textbf{h}) \iff (\textbf{eh})$$
and furthermore in \autoref{thm:mix.char} that also
$$(\textbf{def}) \iff (\emptyset\textrm{-}\textbf{def})$$
holds. Finally, \autoref{ex:mc.no.def} shows that we have
$$(\textbf{h}) \centernot\implies (\emptyset\textrm{-}\textbf{def})$$
in $\mathcal{K}_{0,p}$. This completes the proof.
\end{proof}

\section*{Acknowledgements}

The authors would like to extend their thanks to the Nesin Mathematics Village (\url{https://matematikkoyu.org/eng/}) for its hospitality during the research visit on which this work was begun.  

\def\bibfont{\footnotesize}
\bibliographystyle{alpha}
\bibliography{bibliography}

\begin{thebibliography}{JSW15}

\bibitem[EP05]{EP05}
Antonio~J. Engler and Alexander Prestel.
\newblock {\em Valued fields}.
\newblock Springer Monographs in Mathematics. Springer-Verlag, Berlin, 2005.

\bibitem[FJ15]{Fehm-Jahnke15}
Arno Fehm and Franziska Jahnke.
\newblock On the quantifier complexity of definable canonical henselian
  valuations.
\newblock {\em Mathematical Logic Quarterly}, 61(4-5):347--361, 2015.

\bibitem[FP11]{Fehm-Paran11}
Arno Fehm and Elad Paran.
\newblock Non-ample complete valued fields.
\newblock {\em Int. Math. Res. Not. IMRN}, (18):4135--4146, 2011.

\bibitem[Hod97]{hod}
Wilfrid Hodges.
\newblock {\em A shorter {M}odel {T}heory}.
\newblock Cambridge University Press, Cambridge, 1997.

\bibitem[Hon14]{Hon14}
Jizhan Hong.
\newblock Definable non-divisible {H}enselian valuations.
\newblock {\em Bull. Lond. Math. Soc.}, 46(1):14--18, 2014.

\bibitem[JK15a]{Jahnke-Koenigsmann15}
Franziska Jahnke and Jochen Koenigsmann.
\newblock Defining coarsenings of valuations.
\newblock Preprint, available on ArXiv:1501.04506 [math.LO], to appear in
  Proceedings of the Edinburgh Mathematical Society, 2015.

\bibitem[JK15b]{Jahnke-Koenigsmann14}
Franziska Jahnke and Jochen Koenigsmann.
\newblock Uniformly defining $p$-henselian valuations.
\newblock {\em Annals of Pure and Applied Logic}, 166:741--754, 2015.

\bibitem[Joh15]{WJ}
Will Johnson.
\newblock On dp-minimal fields.
\newblock Preprint, available on ArXiv:1507.02745 [math.LO], 2015.

\bibitem[JSW15]{JSW}
Franziska Jahnke, Pierre Simon, and Erik Walsberg.
\newblock Dp-minimal valued fields.
\newblock Preprint, available on ArXiv:1507.03911 [math.LO], 2015.

\bibitem[Koe94]{Koe94}
Jochen Koenigsmann.
\newblock Definable valuations.
\newblock Preprint, 1994.

\bibitem[Kru15]{Kru15}
Krzysztof Krupi{\'n}ski.
\newblock Superrosy fields and valuations.
\newblock {\em Ann. Pure Appl. Logic}, 166(3):342--357, 2015.

\bibitem[Kuh14]{Kuhlmann}
Franz-Viktor Kuhlmann.
\newblock The algebra and model theory of tame valued fields.
\newblock Preprint, available on arXiv:1304.0194v2 [math.AC], to appear in J.
  reine angew. Math., 2014.

\bibitem[PD11]{PD}
Alexander Prestel and Charles~N. Delzell.
\newblock {\em Mathematical logic and model theory}.
\newblock Universitext. Springer, London, 2011.
\newblock A brief introduction, Expanded translation of the 1986 German
  original.

\bibitem[Pre14]{Pr14}
Alexander Prestel.
\newblock Definable henselian valuation rings.
\newblock Preprint, available on ArXiv:1401.4813 [math.AC], 2014.

\bibitem[PZ78]{Prestel-Ziegler78}
Alexander Prestel and Martin Ziegler.
\newblock Model-theoretic methods in the theory of topological fields.
\newblock {\em J. Reine Angew. Math.}, 299 (300):318--341, 1978.

\bibitem[TZ12]{TZ}
Katrin Tent and Martin Ziegler.
\newblock {\em A course in model theory}, volume~40 of {\em Lecture Notes in
  Logic}.
\newblock Association for Symbolic Logic, La Jolla, CA, 2012.

\end{thebibliography}
\end{document}